\documentclass[a4paper,11pt]{article}

{\makeatletter
 \gdef\xxxmark{%
   \expandafter\ifx\csname @mpargs\endcsname\relax 
     \expandafter\ifx\csname @captype\endcsname\relax 
       \marginpar{{xxx}}
     \else
       {xxx} 
     \fi
   \else
     {xxx} 
   \fi}
 \gdef\xxx{\@ifnextchar[\xxx@lab\xxx@nolab}
 \long\gdef\xxx@lab[#1]#2{{\bf [\xxxmark #2 ---{\sc #1}]}}
 \long\gdef\xxx@nolab#1{{\bf  [\xxxmark #1]}}
 \long\gdef\xxx@lab[#1]#2{}\long\gdef\xxx@nolab#1{}%
}

\usepackage{a4wide}
\usepackage[ansinew]{inputenc}
\usepackage{amsmath, amsfonts} 
\usepackage{amsthm, amssymb}
\usepackage[english]{babel}
\usepackage{lmodern}
\usepackage{graphicx}
\usepackage{subfigure}

\newtheorem{theorem}{Theorem}
\newtheorem{innercustomthm}{Theorem}
\newenvironment{customthm}[1]
  {\renewcommand\theinnercustomthm{#1}\innercustomthm}
  {\endinnercustomthm}
\newtheorem{proposition}{Proposition}
\newtheorem{conjecture}{Conjecture}
\newtheorem{innercustomconj}{Conjecture}
\newenvironment{customconj}[1]
  {\renewcommand\theinnercustomconj{#1}\innercustomconj}
  {\endinnercustomthm}
\newtheorem{corollary}{Corollary}
\newtheorem{lemma}{Lemma}

\theoremstyle{definition}

\newtheorem{remark}{Remark}

\usepackage{verbatim}
\usepackage{array}
\usepackage{float}
\usepackage{subfigure}
\usepackage{color}

\title{The Four Bars Problem}
\author{Alexandre Mauroy\thanks{Luxembourg Centre for Systems Biomedicine, University of Luxembourg, Luxembourg (alexandre.mauroy@uni.lu). This work was performed while A. Mauroy was with the Department of Electrical Engineering and Computer Science at the University of Li\`ege, Belgium and held a return grant from the Belgian Science Policy (BELSPO).} \and Perouz Taslakian\thanks{Universit\'e Libre de Bruxelles, Belgium (perouz.taslakian@ulb.ac.be)} \and Stefan Langerman\thanks{Directeur de Recherches du F.R.S.-FNRS, Universit\'e Libre de Bruxelles, Belgium (stefan.langerman@ulb.ac.be)} \and Rapha\"{e}l Jungers\thanks{ICTEAM Institute,
Universit\'e Catholique de Louvain, Belgium (raphael.jungers@uclouvain.be)}}

\begin{document}
\maketitle

\begin{abstract}
A four-bar linkage is a mechanism consisting of four rigid bars which are joined by their endpoints 
in a polygonal chain and which can rotate freely at the joints (or \emph{vertices}). 
We assume that the linkage
lies in the 2-dimensional plane so that one of the bars is held horizontally fixed.
In this paper we consider the problem of reconfiguring a four-bar linkage using an operation called a \emph{pop}. Given a polygonal cycle, a pop reflects a vertex across the line defined by its two adjacent vertices along the polygonal chain. 
Our main result shows that for certain conditions on the lengths of the bars of the four-bar linkage, the neighborhood of any configuration that can be reached by smooth motion can also be reached by pops. The proof relies on the fact that pops are described by a map on the circle with an irrational number of rotation.
\end{abstract}

\section{Introduction}

Mechanical linkage chains are important frameworks in machines and their motion has been investigated extensively. 
A classical example of a well-studied framework, and perhaps the simplest, 
is the four-bar linkage.
Four-bar linkages (or \emph{quadilaterals}), often also referred to as \emph{three-bar linkages} when one of the bars is fixed~\cite{roberts-75}, have been studied in the field
of kinematics, where they are mainly used to generate curves 
by converting one type of motion (e.g. circular) into another (e.g. linear)~\cite{mccarthy-06}.
Linkages have also been of interest to mathematicians, 
who developed different tools and techniques
to understand the motions of these frameworks~\cite{benoist2004iteration,darboux1879,kapovich-96,complexGeometry-13,mermoud-00,muller-96}.

In the combinatorial geometry world, Paul Erd{\H o}s initiated the study of polygonal linkage reconfiguration with his question on
flipping the pockets of a polygon. Given a simple polygon (i.e. having no self intersection) in the Euclidean plane,
a \emph{pocket} is a maximal
connected region exterior to the polygon and interior to its convex hull.
Such a pocket is bounded by one edge of the convex hull of the polygon, called the \emph{pocket lid}, and a subchain of the polygon,
called the \emph{pocket subchain} (Figure \ref{flip}(a)).
A \emph{pocket flip} (or simply \emph{flip}) is the operation of reflecting
the pocket subchain through the line extending the pocket lid (Figure \ref{flip}(b)).
The result is a new, simple polygon of larger area
with the same edge lengths as the original polygon.
A convex polygon has no pocket and hence admits no flip.
In 1935, Erd\H{o}s (at the young age of 22) asked if it is possible to convexify a polygon with a finite number of flips~\cite{erdos-35}.
The answer to this question was given four years later by Bella Nagy~\cite{nagy-39}.
Although Nagy proved that every simple polygon can indeed be convexified by a finite number of flips,
his proof was found to contain a flaw many years later.
In fact, the many different proofs of Erd\H{o}s conjecture turn out to have a long trail of interesting stories, 
which were summarized in the work of Demaine et al.\cite{Demaine-08}.
Inspired by the notion of flips, many variations of combinatorial polygonal reconfigurations were defined and studied~\cite{ballinger-03,toussaint-05}.
To list of a few, a \emph{deflation} is the inverse operation of a flip, 
and a polygon admits no deflation if no such operation results in a simple polygon.
It was shown that, unlike in the case of flips, there exist polygons that deflate infinitely~\cite{fevens-01} and that the limit of
any  infinite deflation sequence for such polygons is unique~\cite{hubard-10}.
Other studied variations are \emph{flipturns}~\cite{Gru-01, Gru-95} and \emph{mouth flips}~\cite{millet-94}. 
One particular case of the pocket flip is an operation called \emph{pop}.
A pop reflects a single vertex $v$ of a polygon across the line defined by the vertices adjacent to $v$ (Figure \ref{pop}).
The question of whether 
polygons can be convexified by a (finite) series of pops, under various
intersection restrictions and definitional variants
is studied by Aloupis et al.~\cite{aloupis-07}. 
Dumitrescu and Hilscher show that not every polygon can be convexified with pops~\cite{dumitrescu-10}.
However, their counterexample is highly degenerate (all vertices of the polygon lie on two orthogonal lines). 
This prompts the natural question of whether only degenerate polygons exhibit this behavior. 
\begin{figure}[h]
\begin{center}
\subfigure[]{\includegraphics[width=4cm]{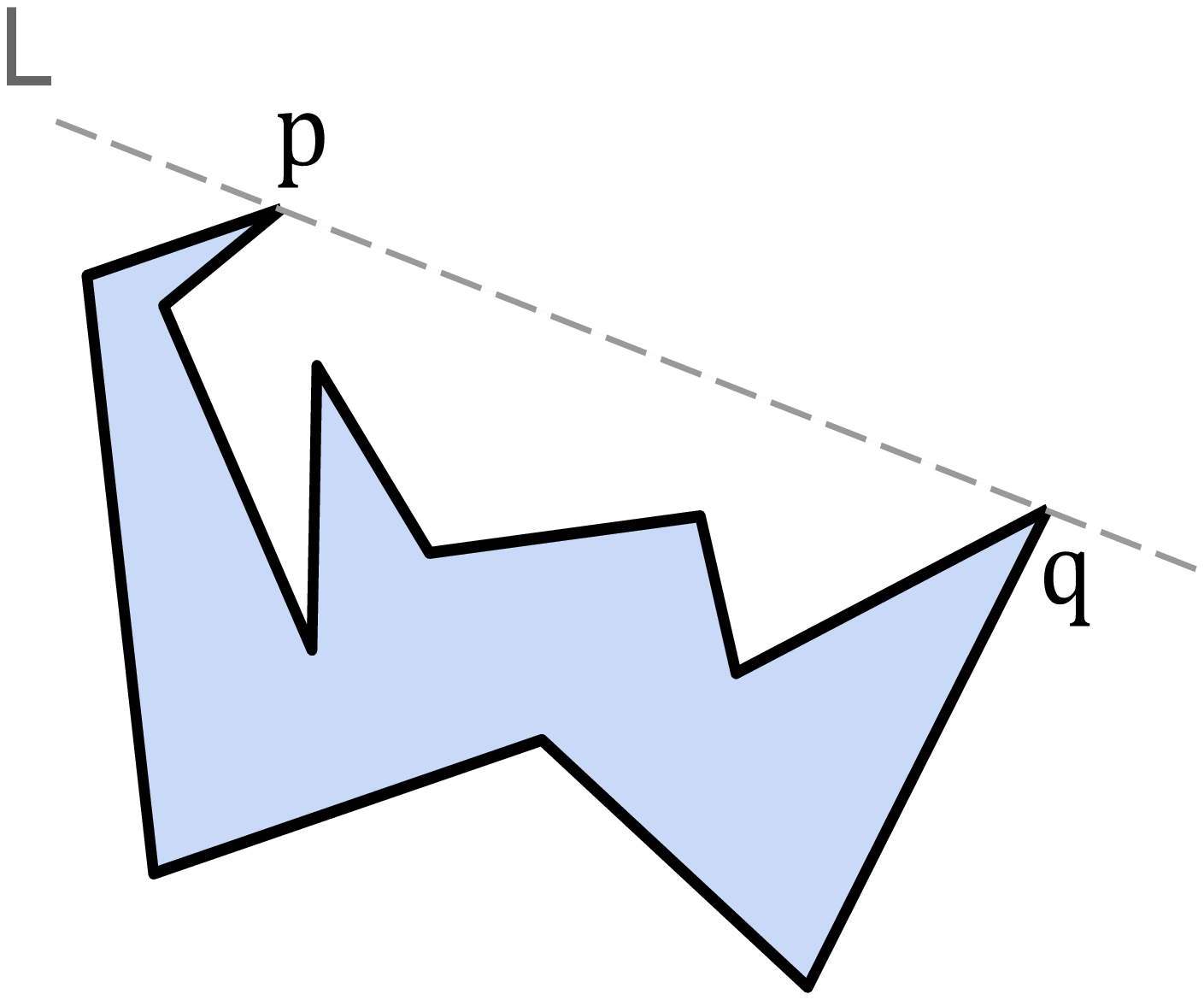}}
\hspace{1cm}
\subfigure[]{\includegraphics[width=4cm]{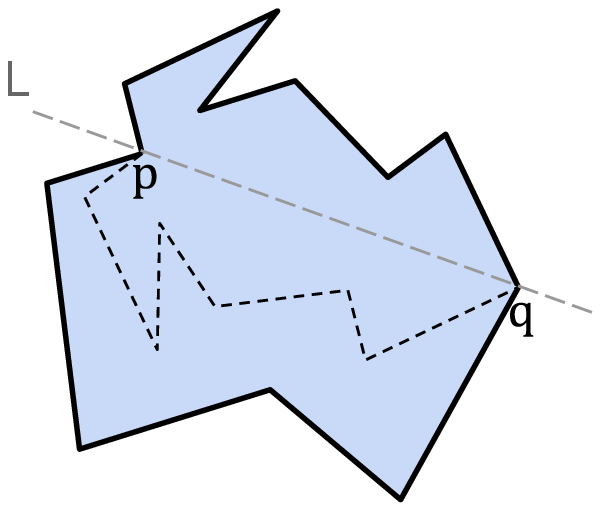}}
\caption{(a) The pocket is bounded by the segment $pq$ (the pocket lid) and a subchain of the polygon (the pocket subchain from $p$ to $q$). (b) The pocket flip is the reflection of the pocket subchain through the line $L$ extending the pocket lid.}
\label{flip}
\end{center}
\end{figure}
\begin{figure}[h]
\begin{center}
\subfigure[]{\includegraphics[height=3cm]{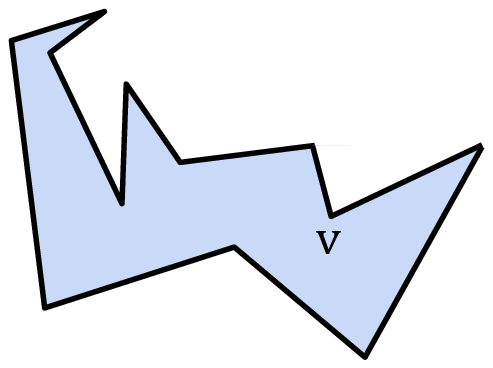}}
\hspace{1cm}
\subfigure[]{\includegraphics[height=2.9cm]{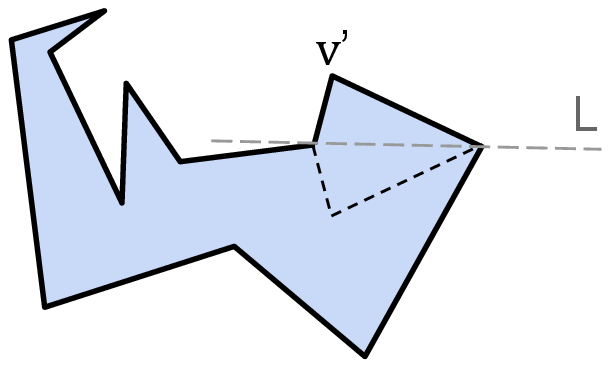}}
\caption{A pop is the reflection of a vertex $v$ through the line $L$ defined by the two vertices adjacent to $v$.}
\label{pop}
\end{center}
\end{figure}

In this paper, we provide a first hint that every nondegenerate polygon can be convexified with pops.
We focus on the simplest polygonal chain --- the four-bar linkage, and in fact on an even simpler system, 
where one bar is assumed to be fixed. Equivalently, we restrict the pop operations to be applied to only two vertices 
among the four (i.e., the two vertices that are not adjacent to the fixed bar). 
This restriction can later be removed as popping two opposite vertices results in the mirror image of the four-bar linkage. 
Because applying a pop twice in a row to the same vertex leaves the linkage unchanged, we are left with the analysis of a single sequence of pops, 
alternating between the two mobile vertices. 
Let $\cal C$ and $\cal C'$ be two configurations of the same four-bar linkage.
We say that we can \emph{reach a neighborhood} of $\cal C'$ if for any $\cal C$
and every $\varepsilon > 0$,
there exists a sequence of pops such that, when applied to $\cal C$,
will result in a configuration where every vertex is at distance at most $\varepsilon$ 
to its position in $\cal C'$. 
We believe that even under the strong restriction of fixing one of the bars of a four-bar linkage, the neighborhood of the full set of configurations that are reachable by 
continuous motion (Conjecture \ref{conjecture} below) can be reached by a sequence of pops, as suggested in Figure \ref{simu_four_bars}. 
We prove this under additional assumptions on the numerical values of the lengths of the bars. 
By doing so, we provide a first step towards a complete understanding of the pop operation on general polygons.
\begin{figure}[h]
\begin{center}
\includegraphics[width=6cm,angle=-1.3,origin=c]{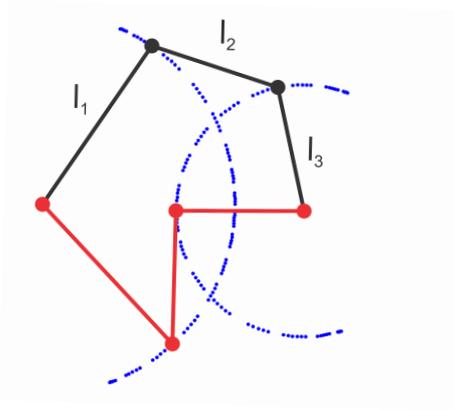}
\caption{A typical behavior of the pop iteration on a four-bar linkage with a fixed bar (not shown in this figure). 
All the configurations seem to be reached by successive pops of a pair of vertices. 
The black chain (upper hull) is the original configuration of the three bars, 
while the red one below is the chain after 166 pops.}
\label{simu_four_bars}
\end{center}
\end{figure}

\begin{conjecture}
\label{conjecture}
For almost every four-bar linkage, 
the neighborhood of any configuration that can be reached by smooth motion, 
can also be reached by 
a sequence of alternating pops of two adjacent vertices.
\end{conjecture}

It is easily seen that, when the bars have the same length, the four-bar linkage recovers its initial configuration after 6 pops. This shows that there exist cases where almost all configurations cannot be reached by a sequence of pops. Conjecture \ref{conjecture} implies that these situations are not generic: if we consider that a four-bar linkage is described by four parameters in $\mathbb{R}^4_{>0}$ (i.e. the lengths of the bars), these situations correspond to a zero measure set of parameters in $\mathbb{R}^4_{>0}$. Note also that a four-bar linkage (with a fixed bar) has only one degree of freedom, so that feasible configurations that can be reached by smooth motion are represented by a one-dimensional manifold. Conjecture \ref{conjecture} states that an infinite sequence of pops densely fills this one-dimensional manifold of feasible configurations, that is, for any small value $\epsilon > 0$ and from any configuration, 
one can reach the $\epsilon$-neighborhood of any other configuration with a sequence of pops.

Our key idea is to focus on the properties of the two-dimensional map describing the dynamics of four-bar linkages. The one-dimensional manifold of feasible configurations (which depends on the particular linkage, via the four parameters) is invariant under this map. We show that the map restricted to this manifold is topologically equivalent to an orientation-preserving map of the circle. We then prove that the rotation number of this map is irrational for almost all admissible sets of parameters, so that an orbit of the map is dense in the invariant set.

After the submission of this paper, the authors became aware of a recent unpublished manuscript studying periodicity properties of the four bar linkage \cite{Izmestiev}. The results of that paper might be applied to Conjecture \ref{conjecture}, thereby complementing our present study.

The rest of the paper is organized as follows. In Section~\ref{4bars}, we present and discuss the main result. The two-dimensional map describing the four-bar linkage is derived in Section~\ref{2D map} and the main result is proved in Section~\ref{orbits}. Finally, concluding remarks are given in Section~\ref{conclu}.


\section{Four-bar linkage}
\label{4bars}

As mentioned above, in this paper, we consider a four-bar linkage that is composed of three consecutive bars numbered $1$, $2$, and $3$ (called respectively \emph{input link}, \emph{floating link}, and \emph{output link}), and a fourth bar that is held horizontally fixed (called ground link). The four-bar linkage is shown in Figure \ref{4bars_pops}(a). We denote the lengths of bars 1, 2, and 3 by $l_1$, $l_2$ and $l_3$, respectively, and the length of the fourth fixed bar by $L$. Note that we must have
\begin{equation}
\label{bounds}
\textrm{max}\{0,l_1-l_2-l_3,-l_1+l_2-l_3,-l_1-l_2+l_3\} < L < l_1+l_2+l_3\,.
\end{equation}
We further assume that the bars are allowed to intersect.
\begin{figure}[hb]
\begin{center}
\subfigure[]{\includegraphics[width=3cm]{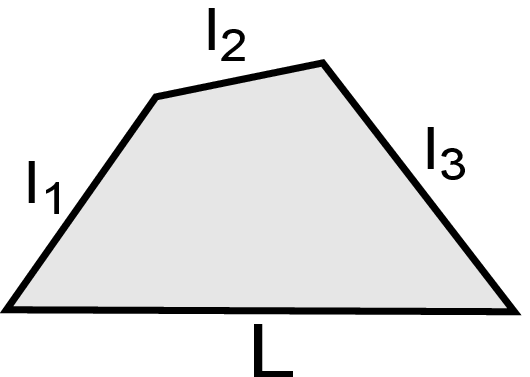}}
\hspace{1cm}
\subfigure[]{\includegraphics[width=3cm]{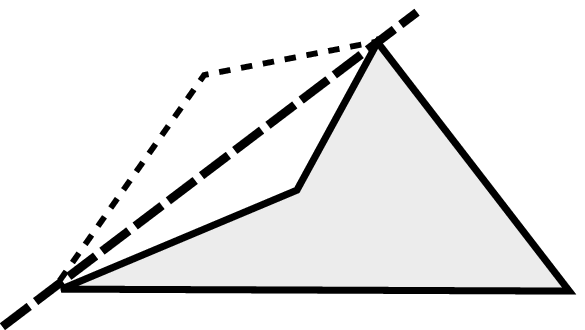}}
\hspace{1cm}
\subfigure[]{\includegraphics[width=3cm]{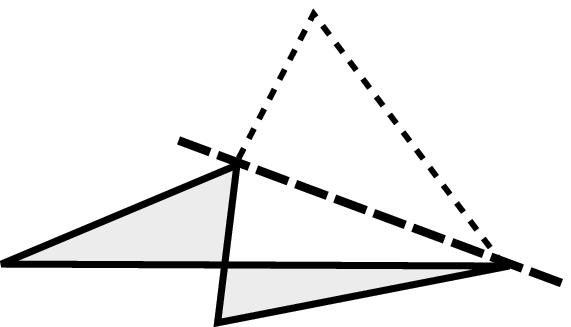}}
\caption{(a) A four-bar linkage. (b) Its configuration after popping bars $1$--$2$. 
(c) Configuration after popping bars $2$--$3$.} 
\label{4bars_pops}
\end{center}
\end{figure}

\subsection{Smooth motion and feasible configurations}

Depending on the length of the bars, the four-bars linkage exhibits different types of smooth motion, and is characterized by different types of feasible configurations. The input and output links either fully rotate with respect to the fixed bar by $2\pi$ (i.e. \emph{crank} motion) or move only in a limited range of angles (i.e. \emph{rocker} motion). 
It is known \cite{mccarthy-06} that the type of movement is determined by the sign of the terms
\begin{eqnarray*}
T_1 & \triangleq & -l_1+l_2-l_3+L \,, \\
T_2 & \triangleq & -l_1-l_2+l_3+L \,, \\
T_3 & \triangleq & -l_1+l_2+l_3-L \,. 
\end{eqnarray*}
It follows that eight different situations can be observed, each of which corresponding to a specific combination of the signs of $T_1$, $T_2$, and $T_3$. 
In four cases, the so-called Grashof condition $T_1T_2T_3>0$ is satisfied and 
not all feasible configurations can be reached by smooth motion, that is, the configuration space is disconnected. 
In the other cases (with $T_1T_2T_3<0$), the configuration space is connected and all configurations can be reached by smooth motion.

\subsection{Motion induced by pops}

The pop operation is applied on two vertices of the framework. In what follows, we consider that these two vertices do not belong to the ground (fixed) link, so that a pop will move either the bar pair $1$-$2$, or the pair $2$-$3$. A sequence of pops alternates between popping these two pairs (see Figure \ref{4bars_pops}).

Our main result is a proof of Conjecture \ref{conjecture} in the (non-Grashof) situation $T_1>0$, $T_2>0$, and $T_3<0$ ($0\pi$ \emph{double-rocker}, see \cite{mccarthy-06}).
\begin{theorem}
\label{main_theorem}
For almost all four-bar linkages that satisfy the conditions $T_1>0$, $T_2>0$, $T_3<0$, and
\begin{equation}
\label{cond_main_theorem}
l_2\leq \min\{l_1,l_3\} \textrm{ or } l_2\geq \max\{l_1,l_3\}\,,
\end{equation}
the neighborhood of any configuration that can be reached by smooth motion can be reached by a sequence of pops by moving the bars $1$-$2$ and $2$-$3$.
\end{theorem}

We do not have a clear geometric interpretation of why Condition \eqref{cond_main_theorem} is important for our theorem. It appears that we need it in our proof in order to ensure some monotonicity properties of the system.
We restate Theorem~\ref{main_theorem} more precisely in Theorem~\ref{theo_dense} (Section~\ref{orbits}). The proof is given in Section \ref{subsec:proof_main}.

On one hand, Theorem \ref{main_theorem} proves Conjecture \ref{conjecture} in the restrictive case of a single configuration, but on the other hand it proves a stronger statement, in that the sequence of pops does not involve the ground (fixed) link.  If one relaxes this additional requirement, then Conjecture \ref{conjecture} is also shown to be true in the three other non-Grashof situations. We have the following corollary of Theorem \ref{main_theorem}.
\begin{corollary}
\label{coroll}
For almost all four-bar linkages that satisfy $T_1 T_2 T_3 < 0$ and
\begin{equation}
\label{cond_coroll}
\max\{l_2,L\} \leq \min\{l_1,l_3\} \textrm{ or } \min\{l_2,L\} \geq \max\{l_1,l_3\}\,,
\end{equation}
the neighborhood of any configuration that can be reached by smooth motion can also be reached by a sequence of pops alternating between two particular vertices (chosen among the four pairs of adjacent vertices).
\end{corollary}
\begin{proof}
The proof is based on the fact that all non-Grashof situations can be obtained by reassigning the labels of the bars, that is, by choosing for instance another ground (fixed) bar. More precisely, if $T_1 T_2 T_3 < 0$, there exists a cyclic permutation $\sigma$ such that the terms $T'_1$, $T'_2$, $T'_3$ obtained with the values $(l'_1,l'_2,l'_3,L')=\sigma(l_1,l_2,l_3,L)$ satisfy $T'_1>0$, $T'_2>0$, and $T'_3<0$. In addition, it is easy to show that the values $(l'_1,l'_2,l'_3,L')$ satisfy the condition \eqref{cond_main_theorem} if \eqref{cond_coroll} holds. Then the result follows from Theorem \ref{main_theorem} applied to the four-bar linkage with lengths $l'_1$, $l'_2$, $l'_3$, and $L'$.
\end{proof}

Theorem \ref{main_theorem} and Corollary \ref{coroll} partially solve Conjecture \ref{conjecture}, in the non-Grashof case. The conjecture is not solved here in the Grashof situation (i.e. disconnected configuration space). For this case, we have the following result.
\begin{proposition}
\label{prop_non_ergo}
If $T_1<0$, $T_2>0$, and $T_3<0$, then it is not possible to reach the neighborhood of all configurations of a four-bar linkage by a sequence of pops.
\end{proposition}
\begin{proof}
The set of feasible configurations must contain a subset $S_{up}$ of configurations where the bars $1$-$2$-$3$ lie above the fixed bar, and a subset $S_{down}$ of configurations where the bars $1$-$2$-$3$ lie below the fixed bar. If the bars can reach the neighborhood of all possible configurations by iterative pops, then the four-bar linkage in a configuration of $S_{up}$ can reach a configuration of $S_{down}$. This implies that a configuration of transition between $S_{up}$ and $S_{down}$, where bar 1 is above the fixed bar and bar 3 is below the fixed bar (or vice-versa), must be feasible. This particular configuration can be obtained in three ways: (a) bar $2$ lies on the left of the fixed bar (Figure \ref{transition}(a)); (b) bar $2$ lies on the right of the fixed bar (Figure \ref{transition}(b)); (c) bar $2$ intersects the fixed bar (Figure \ref{transition}(c)). In situation (a), it is clear that $l_1+L<l_2+l_3$, which contradicts $T_3<0$. In situation (b), 
we have $l_3+L<l_1+l_2$, which contradicts $T_2>0$. In situation (c), it follows from the triangle inequality that
\begin{eqnarray*}
l_1 & \leq & L_a+l_{2a} \,,\\
l_3 & \leq & L_b+l_{2b} \,.
\end{eqnarray*}
Summing the two inequalities and using $L=L_a+L_b$ and $l_2=l_{2a}+l_{2b}$, we obtain
\begin{equation*}
l_1+l_3 \leq L + l_2\,,
\end{equation*}
which contradicts $T_1<0$. It follows that the transition configuration is not feasible, and the subset $S_{down}$ cannot be reached by pops starting from a configuration of $S_{up}$.
\end{proof}

\begin{figure}[t]
\begin{center}
\subfigure[]{\includegraphics[width=3cm]{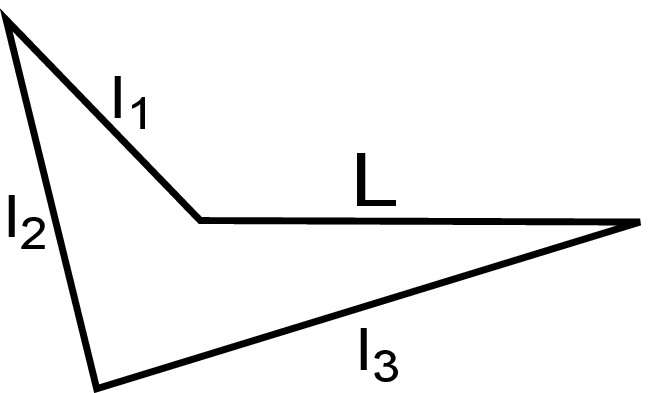}}
\hspace{1cm}
\subfigure[]{\includegraphics[width=3cm]{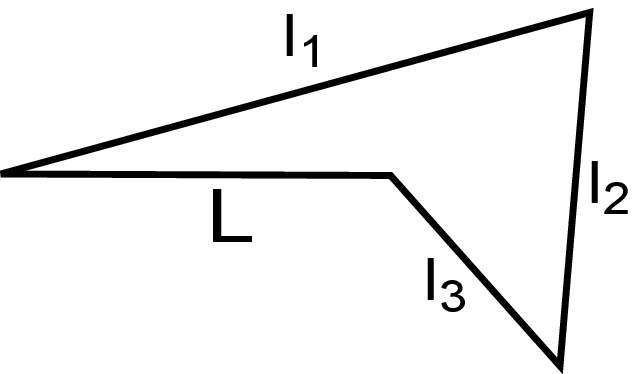}}
\hspace{1cm}
\subfigure[]{\includegraphics[width=3cm]{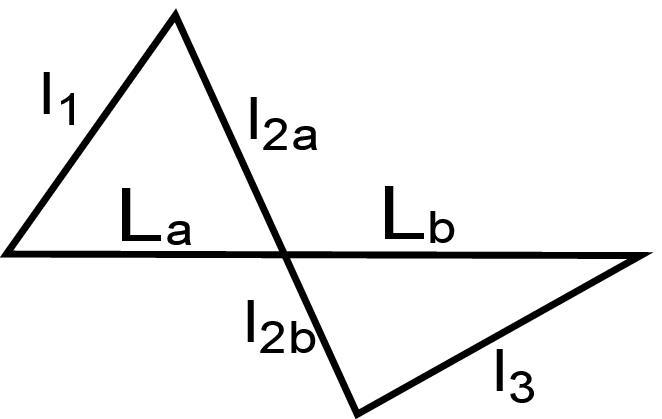}}
\caption{The three cases in the proof of Proposition \ref{prop_non_ergo}, for the possible configurations such that bar $1$ is above the extremities and bar $3$ is below the extremities. 
In (c), we denote $L=L_a+L_b$ and $l_2=l_{2a}+l_{2b}$.}
\label{transition}
\end{center}
\end{figure}
Proposition \ref{prop_non_ergo} is not a counter-example to Conjecture \ref{conjecture}. Although a sequence of pops cannot reach all the configurations of $S_{up} \cup S_{down}$, it might reach all the configurations of either $S_{up}$ or $S_{down}$. These configurations correspond to all the configurations that can be obtained by smooth motion. Note that in the three other situations satisfying the Grashof condition, numerical simulations suggest that the full configuration space can be reached by a sequence of pops, even though it cannot be connected by smooth motion.

\section{Derivation of a two-dimensional map}
\label{2D map}

In this section, we show that the pop operation can be described by a two-dimensional discrete time map.

\subsection{Two-dimensional map}

The four-bar linkage can be described with two angles: the counterclockwise-turning angle $\theta_1\in (-\pi,\pi]$ from bar 1 to bar 2 and the counterclockwise-turning 
angle $\theta_2 \in (-\pi,\pi]$ from bar 2 to bar 3 (see Figure \ref{schema}). 
That is, $\theta_1$ (resp. $\theta_2$) is negative when it is measured clockwise from bar 1 to bar 2 (resp. from bar 2 to bar 3). It is clear that each pair of angles $(\theta_1,\theta_2)$ corresponds to one and only one configuration of the bars.

\begin{figure}[h]
\begin{center}
\includegraphics[width=10cm]{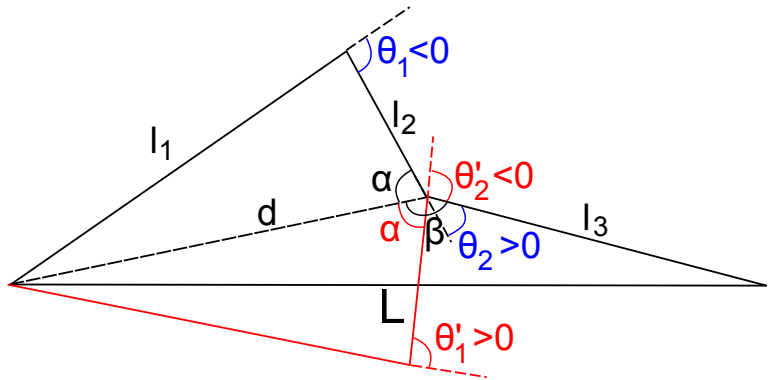}
\caption{The system is described with the two angles $\theta_1$ and $\theta_2$. Through popping bars 1-2 (in red), the angles are modified according to $\theta_1'=-\theta_1$ and $\theta_2'=\theta_2+\textrm{sign}(\theta_1) \, 2 \alpha$.}
\label{schema}
\end{center}
\end{figure}

Popping bars 1-2 produces the new angles
\begin{eqnarray*}
\theta_1' & = & -\theta_1 \\
\theta_2' & = & \langle \theta_2+\textrm{sign}(\theta_1) \, 2 \alpha\rangle
\end{eqnarray*}
with $\alpha \geq 0$ and where we have defined the operation $\langle x\rangle=((x +\pi) \bmod 2\pi) -\pi$, which ensures that $x \in (-\pi,\pi]$. In addition, we have
\begin{eqnarray}
d^2 & = & l_1^2+l_2^2+2 l_1 l_2 \cos \theta_1  \label{equa_d}\\
l_1^2 & = & l_2^2+d^2-2l_2  d \cos \alpha
\end{eqnarray}
so that
\begin{equation}
\label{equa_alpha}
\alpha=\arccos \left(\frac{l_2+l_1 \cos \theta_1}{\sqrt{l_1^2+l_2^2+2 l_1 l_2 \cos\theta_1}}\right)\,.
\end{equation}
It follows that the above equations can be rewritten as
\begin{equation}
\label{syst}
\begin{array}{rcl}
\theta_1' & = & -\theta_1 \\
\theta_2' & = & \displaystyle \left \langle\theta_2+\textrm{sign}(\theta_1) \, 2 \arccos \left(\frac{l_2+l_1 \cos \theta_1}{\sqrt{l_1^2+l_2^2+2 l_1 l_2 \cos\theta_1}}\right) \right \rangle
\end{array}
\triangleq \langle H_{12}(\theta_1,\theta_2) \rangle
\end{equation}
with $H_{12}(\theta_1,\theta_2):(-\pi,\pi] \to \mathbb{R}$. The discrete time map \eqref{syst} describes the change of angles induced by popping  bars 1-2. For bars 2-3, it follows on similar lines that
\begin{equation}
\label{systb}
\begin{array}{rcl}
\theta_1' & = &  \displaystyle \left \langle \theta_1+\textrm{sign}(\theta_2) \, 2 \arccos \left(\frac{l_2+l_1 \cos \theta_2}{\sqrt{l_1^2+l_2^2+2 l_1 l_2 \cos\theta_2}}\right) \right \rangle \\
\theta_2' & = & -\theta_2
\end{array}
\triangleq \langle H_{23}(\theta_1,\theta_2) \rangle
\end{equation}
with $H_{23}:(-\pi,\pi] \to \mathbb{R}$. The alternate pops of bars 1-2 and bars 2-3 are described by the composition $\langle H_{23}\rangle \circ \langle H_{12} \rangle$.

We have the following preliminary result.
\begin{proposition}
If the bars are identical (i.e. $l_1=l_2=l_3$), then they recover the initial configuration after 6 pops.
\end{proposition}
\begin{proof}
If $l_1=l_2=l_3$, we have $\langle H_{12} \rangle(\theta_1,\theta_2) = (-\theta_1,\theta_1+\theta_2)$ and $\langle H_{23} \rangle(\theta_1,\theta_2) = (\theta_1+\theta_2,-\theta_2)$. It is easy to see that $\langle H_{23}\rangle  \circ \langle H_{12} \rangle$ is periodic with period 3.
\end{proof}

This result shows that there exist cases where all configurations cannot be reached by a sequence of pops. Conjecture \ref{conjecture} implies that these particular situations correspond to a zero measure set of parameters.

\subsection{Invariant set}

The two-dimensional maps \eqref{syst} and \eqref{systb} describe the behavior of all four-bar linkages characterized by the lengths $l_1$, $l_2$, and $l_3$. In order to consider the behavior of a unique four-bar linkage, one has to treat the length $L$ of the fixed bar (not used to derive \eqref{syst}-\eqref{systb}) as an additional constraint. In this case the maps \eqref{syst} and \eqref{systb} are restricted to an invariant one-dimensional manifold, which corresponds to the set of all admissible configurations $(\theta_1,\theta_2)$ of the four-bar linkage. We have
\begin{equation*}
L^2=d^2+l_3^2-2 d l_3 \cos \beta
\end{equation*}
and one can verify on Figure \ref{schema} that $\beta=2\pi \pm (\theta_2 +\textrm{sign}(\theta_1) \alpha)$. Using \eqref{equa_alpha}, we obtain
\begin{equation*}
L^2=d^2+l_3^2+2 d l_3 \cos \left[\theta_2 +\textrm{sign}(\theta_1) \arccos \left(\frac{l_2+l_1 \cos \theta_1}{d}\right)\right]
\end{equation*}
or equivalently, using basic trigonometry,
\begin{equation*}
\begin{split}
L^2=d^2 & +l_3^2+ 2 l_2 l_3 \cos \theta_2 + 2 l_1 l_3 \cos \theta_1 \cos \theta_2 \\
& - \textrm{sign}(\theta_1) 2 l_3 \sin \theta_2 \sqrt{d^2-l_2^2-l_1^2 \cos^2 \theta_1-2 l_1 l_2 \cos \theta_1}\,.
\end{split}
\end{equation*}
Finally, it follows from \eqref{equa_d} that
\begin{equation*}
L^2 = l_1^2+l_2^2+l_3^2+2l_1l_2 \cos \theta_1 +2 l_2 l_3 \cos \theta_2+ 2 l_1 l_3 \cos \theta_1 \cos \theta_2 - \textrm{sign}(\theta_1) 2 l_1 l_3 \sin \theta_2 |\sin \theta_1| 
\end{equation*}
or
\begin{equation}
\label{invariant_set}
L^2 = l_1^2+l_2^2+l_3^2 +2l_1l_2 \cos \theta_1 +2 l_2 l_3 \cos \theta_2 + 2 l_1 l_3 \cos (\theta_1+\theta_2) \triangleq \bar{L}^2(\theta_1,\theta_2)\,.
\end{equation}
This equality defines an invariant one-dimensional manifold
\begin{equation*}
\Gamma(L) = \{(\theta_1,\theta_2) \in (-\pi,\pi]^2 \,|\, \bar{L}(\theta_1,\theta_2)=L \}\,
\end{equation*}
and it is easy to see that all the pairs $(\theta_1,\theta_2)$ are feasible; that is, $\theta_1,\theta_2 \in \Gamma(L)$ is a necessary and sufficient condition for the pair of angles to describe a configuration of the four bars for a given value $L$. We note that $\Gamma(L)$ is reduced to a unique point when $L$ is equal to one of the bounds given in \eqref{bounds}: $\Gamma(L)=\{(0,0)\}$ when $L=l_1+l_2+l_3$; $\Gamma(L)=\{(\pi,0)\}$ when $L=l_1-l_2-l_3$; $\Gamma(L)=\{(0,\pi)\}$ when $L=-l_1-l_2+l_3$; and $\Gamma(L)=\{(\pi,\pi)\}$ when $L=-l_1+l_2-l_3$.\\

\section{Dense orbits}
\label{orbits}

The main goal of this section is to prove Theorem \ref{main_theorem}. We rely on the fact that the system is described by an orientation-preserving map on the circle, whose orbits are dense in the invariant set of possible configurations (for almost all sets of parameters).

The property that the $\varepsilon$-neighborhood of any configuration can be reached with pops is equivalent to the property that the orbits of the map \eqref{syst}-\eqref{systb} are dense in the set of admissible configurations. We can then restate Theorem \ref{main_theorem} as follows.
\begin{customthm}{1bis}\label{theo_dense}
For almost all parameters $l_1,l_2,l_3,L$ that satisfy
\begin{eqnarray}
(T_1>0) \qquad L & > & l_1-l_2+l_3 \label{T1}\\
(T_2>0) \qquad L & > & -l_1+l_2+l_3 \label{T2} \\
(T_3<0) \qquad L & > & l_1+l_2-l_3 \label{T3}
\end{eqnarray}
and $l_2\leq \min\{l_1,l_3\}$ or $l_2\geq \max\{l_1,l_3\}$, every orbit of the map $\langle H_{23}\rangle  \circ \langle H_{12} \rangle$ (see \eqref{syst}-\eqref{systb}), with an initial condition $(\theta_1,\theta_2)\in \Gamma(L)$, is dense in $\Gamma(L)$.
\end{customthm}

We postpone the proof of Theorem \ref{theo_dense}, which relies on intermediate results summarized in several lemmas.

\begin{remark}[Modulo function]
\label{rem_modulo}
When $L$ satisfies the conditions \eqref{T1}-\eqref{T2}-\eqref{T3}, the maps \eqref{syst} and \eqref{systb} can be defined without the function $\langle \cdot \rangle$, i.e.
\begin{equation}
\label{cond_no_mod}
H_{12}(\Gamma(L)) \subseteq (-\pi,\pi]^2 \qquad H_{23}(\Gamma(L)) \subseteq (-\pi,\pi]^2\,.
\end{equation}
For $(\theta_1,\theta_2) \rightarrow (0,0)$ (i.e. $L=\bar{L}(\theta_1,\theta_2) \rightarrow l_1+l_2+l_3$), we have $H_{12}(\theta_1,\theta_2) \rightarrow  (0,0)$ and $H_{23}(\theta_1,\theta_2)\rightarrow  (0,0)$. In other words, for every $\epsilon>0$, there exists $r>0$ such that $\Gamma(L)\subset B(r)$ for all $l_1+l_2+l_3-\epsilon<L<l_1+l_2+l_3$. Hence, \eqref{cond_no_mod} is satisfied for $\epsilon$ small enough. By continuity of the maps $H_{12}$ and $H_{23}$, it follows that \eqref{cond_no_mod} is satisfied as long as $L$ is large enough so that $\Gamma(L)$ contains no point $(\theta_1^*,\theta_2^*)$ with either $\theta_1^*=\pi$ or $\theta_2^*=\pi$. Since \eqref{invariant_set} implies
\begin{eqnarray}
\label{val_L_pi1}
\bar{L}(\theta_1,\pi) & \leq &\max\{|l_1-l_2+l_3|,|l_1+l_2-l_3|\} \,,\\
\label{val_L_pi2}
\bar{L}(\pi,\theta_2) & \leq &\max\{|-l_1+l_2+l_3|,|l_1-l_2+l_3|\} \,,
\end{eqnarray}
$\Gamma(L)$ does not contain $(\theta_1^*,\theta_2^*)$ if $L$ satisfies \eqref{T1}-\eqref{T2}-\eqref{T3} and \eqref{bounds}. \hfill $\diamond$
\end{remark}

\subsection{Polar-type coordinates}

For given parameters $l_1,l_2,l_3$, consider the set $\Lambda$ of feasible values $L$ that satisfy \eqref{T1}-\eqref{T2}-\eqref{T3}, i.e.
\begin{equation*}
\Lambda = \{ L \in \mathbb{R}^+ | \max\{-l_1+l_2+l_3,l_1-l_2+l_3,l_1+l_2-l_3\} < L < l_1+l_2+l_3 \}
\end{equation*}
and define the set
\begin{equation*}
\Omega = \{ (\theta_1,\theta_2) \in (-\pi,\pi]^2 | \bar{L}(\theta_1,\theta_2) \in \Lambda \} \,.
\end{equation*}
We can introduce a polar-type change of coordinates $g:\Omega \to \Lambda \times (-\pi,\pi]$ yielding the new variables
\begin{equation}
\label{change_var}
(L,\phi) = g(\theta_1,\theta_2) = \left(\bar{L}(\theta_1,\theta_2), \Pi(\theta_1,\theta_2)\right)
\end{equation}
where $\Pi:\mathbb{R}^2 \to (-\pi,\pi]$ is the two-argument atan2 function, i.e. $\Pi(\theta_1,\theta_2)$ is the unique $\phi$ such that $\theta_2/\theta_1=\tan(\phi)$ and such that $|\phi|<\pi/2$ if $\theta_1>0$ and $\pi/2<|\phi|<\pi$ if $\theta_1<0$. Note that through the change of coordinates $g$, the parameter $L$ (which is determined by the angles $(\theta_1,\theta_2)$ for fixed values $l_1,l_2,l_3$) will be considered as a state variable of the two-dimensional system.

The following lemma shows that $g$ is a proper change of variable on $\Omega$.
\begin{lemma}
\label{polar_param}
The map $g:\Omega \to \Lambda \times (-\pi,\pi]$ defined by \eqref{change_var} is bijective (i.e. injective and surjective).
\end{lemma}
\begin{proof}
\emph{Surjectivity.} Consider a pair $(L,\phi) \in \Lambda \times (-\pi,\pi]$ (i.e. $L$ satisfies \eqref{T1}-\eqref{T2}-\eqref{T3}). We show that there exists $(\theta_1,\theta_2)\in \Omega$ such that $g(\theta_1,\theta_2)=(L,\phi)$.\\
The equality $\phi=\Pi(\theta_1,\theta_2)$ implies that $(\theta_1,\theta_2) \in \{ (\gamma\theta_1^0,\gamma\theta_2^0)|\gamma>0\}$ for some $(\theta_1^0,\theta_2^0)\in (-\pi,\pi]^2$. In addition, we have
\begin{equation*}
\bar{L}(\gamma\theta_1^0,\gamma\theta_2^0)=l_1+l_2+l_2
\end{equation*}
for $\gamma=0$ and, according to \eqref{val_L_pi1}-\eqref{val_L_pi2},
\begin{equation*}
\bar{L}(\gamma\theta_1^0,\gamma\theta_2^0)\leq \max\{|-l_1+l_2+l_3|,|l_1-l_2+l_3|,|l_1+l_2-l_3|\}
\end{equation*}
for $\gamma=\min\{\pm \pi/\theta_1^0,\pm \pi/\theta_2^0\}$. Since the function $\gamma \mapsto \bar{L}(\gamma\theta_1^0,\gamma\theta_2^0)$ is continuous and since
\begin{equation*}
\max\{|-l_1+l_2+l_3|,|l_1-l_2+l_3|,|l_1+l_2-l_3|\} < L < l_1+l_2+l_2\,,
\end{equation*}
there exists a pair $(\theta_1,\theta_2)=(\gamma\theta_1^0,\gamma\theta_2^0)$, with $\gamma \in(0,\min\{\pm \pi/\theta_1^0,\pm \pi/\theta_2^0\})$, such that $\bar{L}(\theta_1,\theta_2)=L$.\\
\emph{Injectivity.} Consider $(L_a,\phi_a)=g(\theta_{1a},\theta_{2a})$ and $(L_b,\phi_b)=g(\theta_{1b},\theta_{2b})$, and assume that $(L_a,\phi_a)=(L_b,\phi_b)$. We will show that $(\theta_{1b},\theta_{2b})=(\theta_{1a},\theta_{2a})$. It follows from $\phi_a=\phi_b$ that $(\theta_{1b},\theta_{2b})=\bar{\gamma}(\theta_{1a},\theta_{2a})$, for some $\bar{\gamma}>0$. In addition, we have
\begin{equation}
\label{equa_int}
0=L_b^2-L_a^2=\int_1^{\bar{\gamma}} \frac{d \bar{L}^2}{d\gamma} (\gamma \theta_{1a},\gamma \theta_{2a}) d \gamma \,.
\end{equation}
Next, we will prove that the integrand satisfies
\begin{equation}
\label{inequa_param}
\begin{split}
\frac{d \bar{L}^2}{d\gamma} (\gamma \theta_1,\gamma \theta_2) = & -2 l_1 l_2 \gamma \theta_1 \sin (\gamma \theta_1)- 2 l_2 l_3 \gamma \theta_2 \sin (\gamma \theta_2)  \\
& \qquad -2 l_1 l_3 \gamma (\theta_1+\theta_2) \sin (\gamma(\theta_1+\theta_2)) \leq 0\,,
\end{split}
\end{equation}
where the equality holds only for $(\gamma \theta_1,\gamma \theta_2) \in \{(0,0),(0,\pi),(\pi,0),(\pi,\pi)\}$. If \eqref{inequa_param} holds, \eqref{equa_int} implies $\bar{\gamma}=1$ and therefore $(\theta_{1b},\theta_{2b})=(\theta_{1a},\theta_{2a})$, which implies injectivity.\\
The first two terms of \eqref{inequa_param} are negative for all $(\gamma \theta_1,\gamma \theta_2)\in (-\pi,\pi]$. Thus, it is sufficient to show that $|\gamma \theta_1+\gamma \theta_2| < \pi$, which we do now.\\
Let us first suppose that $\min\{l_1,l_2,l_3\}=l_1$ (the other cases are similar, as we argue below). Condition \eqref{T1} and \eqref{bounds} imply $L^2 > (-l_1+l_2+l_3)^2$ and it follows from \eqref{invariant_set} that
\begin{equation*}
l_1 l_2 (\cos \theta_1+1) + l_2 l_3 (\cos \theta_2-1) +l_1 l_3 (\cos (\theta_1+\theta_2)+1) > 0
\end{equation*}
for all $(\theta_1,\theta_2) \in \Gamma(L)$. Since $l_1 l_2\leq l_2 l_3$ and $l_1 l_3 \leq l_2 l_3$, we have
\begin{equation*}
\cos \theta_1 + \cos \theta_2 + \cos (\theta_1+\theta_2) + 1 > 0
\end{equation*}
or equivalently
\begin{equation}
\label{inequa1}
\cos \left(\frac{\theta_1+\theta_2}{2}\right) \left(\cos \left(\frac{\theta_1-\theta_2}{2}\right) + \cos \left(\frac{\theta_1+\theta_2}{2}\right)\right) > 0 \,.
\end{equation}
In the cases $\min\{l_1,l_2,l_3\}=l_2$ and $\min\{l_1,l_2,l_3\}=l_3$, \eqref{T3} and \eqref{T2}, respectively, lead to the same inequality \eqref{inequa1}.\\
Assume that $|\theta_1+\theta_2| \geq \pi$. This implies $|\theta_1-\theta_2| \leq \pi$ for all $(\theta_1,\theta_2) \in (-\pi,\pi]$ and it follows from $\theta_1 \leq \pi$ that $(\theta_1-\theta_2)/2 \leq \pi-(\theta_1+\theta_2)/2$. These inequalities yield the conditions
\begin{equation*}
\cos \left(\frac{\theta_1+\theta_2}{2}\right) \leq 0 \qquad \cos \left(\frac{\theta_1-\theta_2}{2}\right) \leq 0 \qquad  \left| \cos \left(\frac{\theta_1+\theta_2}{2}\right) \right| \leq \left| \cos \left(\frac{\theta_1-\theta_2}{2}\right) \right|\,.
\end{equation*}
This contradicts \eqref{inequa1}, so that $|\theta_1+\theta_2| < \pi$. Since $\bar{L}(\gamma \theta_1,\gamma \theta_2) \in \Lambda$ satisfies \eqref{T1}-\eqref{T2}-\eqref{T3} for all $\min\{1,\bar{\gamma}\} \leq \gamma \leq \max\{1,\bar{\gamma}\}$, we have $|\gamma \theta_1+\gamma \theta_2| < \pi$ in \eqref{inequa_param}.
\end{proof}

Lemma \ref{polar_param} implies that we can describe the system \eqref{syst}-\eqref{systb} in $(L,\phi)$ coordinates. We obtain the map $\tilde{H}_{12}=g \circ H_{12} \circ g^{-1}$, which is given by
\begin{equation}
\label{map_circle}
\begin{array}{rcl}
L' & = &  L \\
\phi' & = & \Pi \left(-\theta_1, \theta_2+\textrm{sign}(\theta_1) \, 2 \arccos \left(\frac{l_2+l_1 \cos \theta_1}{\sqrt{l_1^2+l_2^2+2 l_1 l_2 \cos\theta_1}}\right) \right) \triangleq f_{12}(L,\phi)
\end{array}
\end{equation}
with $(\theta_1,\theta_2)=g^{-1}(L,\phi)$. Similarly, $\tilde{H}_{23}=g \circ H_{23} \circ g^{-1}$ is given by
\begin{equation}
\label{map_circleb}
\begin{array}{rcl}
L' & = &  L \\
\phi' & = & \Pi \left(\theta_1+\textrm{sign}(\theta_2) \, 2 \arccos \left(\frac{l_2+l_1 \cos \theta_2}{\sqrt{l_1^2+l_2^2+2 l_1 l_2 \cos\theta_2}}\right), -\theta_2 \right) \triangleq f_{23}(L,\phi)\,.
\end{array}
\end{equation}

\subsection{Map on the circle}

It follows from \eqref{map_circle}-\eqref{map_circleb} that the effect of two successive pops can be studied through a one-dimensional map $f_L:\mathbb{S} \to \mathbb{S}$ on the circle, which is parameterized by $L$. We define
\begin{equation}
\label{map_circle_comp}
f_L(\phi)=f(L,\phi)=f_{23}(L,f_{12}(L,\phi))\,,
\end{equation}
where $f_{12}$ and $f_{23}$ are given by \eqref{map_circle} and \eqref{map_circleb}.

Since the circle $\mathbb{S}$ is equipped with a cyclic order, we can denote $\phi_a<\phi_b<\phi_c$ if $\phi_a,\phi_b,\phi_c \in \mathbb{S}$ are distinct and if the arc going from $\phi_a$ to $\phi_c$ passes through $\phi_b$ when it follows the orientation of the circle. A map $f$ on the circle preserves the orientation of $\mathbb{S}$ if $f(\phi_a)<f(\phi_b)<f(\phi_c)$ for all $\phi_a<\phi_b<\phi_c \in \mathbb{S}$. When $f$ is continuous and differentiable, an equivalent condition is $df/d\phi>0$ for almost all $\phi \in \mathbb{S}$.
\begin{lemma}
\label{lem_orient}
Assume \eqref{T1}-\eqref{T2}-\eqref{T3} is satisfied. Then the map $f_L:\mathbb{S} \to \mathbb{S}$ (see \eqref{map_circle_comp}) preserves the orientation of $\mathbb{S}$. \hfill $\diamond$
\end{lemma}
\begin{proof}
The Jacobian matrix of $\tilde{H}_{23} \circ \tilde{H}_{12}=g \circ H_{23} \circ H_{12} \circ g^{-1}$ yields
\begin{equation}
\label{equa_Jacob}
J_{\tilde{H}_{23}}(L',\phi') \, J_{\tilde{H}_{12}}(L,\phi) = J_g(\theta''_1,\theta''_2) \, J_{H_{23}}(\theta'_1,\theta'_2) \, J_{H_{12}}(\theta_1,\theta_2) \, J^{-1}_g(L,\phi) \,,
\end{equation}
where $J_H$ is the Jacobian matrix of $H$, i.e. $(J_H)_{ij}=\partial H_i/\partial x_j$ and where $(\theta'_1,\theta'_2)=H_{12}(\theta_1,\theta_2)$, $(\theta''_1,\theta''_2)=H_{23}(\theta'_1,\theta'_2)$, $(L,\phi)=g(\theta_1,\theta_2)$, and $(L',\phi')=g(\theta'_1,\theta'_2)$. For the sake of clarity, we omit the variables and \eqref{equa_Jacob} implies
\begin{equation}
\label{det_Jacob}
\det(J_{\tilde{H}_{23}}) \, \det(J_{\tilde{H}_{12}}) = \det(J_g) \, \det(J_{H_{23}}) \, \det(J_{H_{12}}) \, \det(J_g)^{-1}\,.
\end{equation}
It follows from \eqref{syst}-\eqref{systb} that
\begin{equation}
\label{det_J_H_tilde}
\det(J_{\tilde{H}_{23}}) \, \det(J_{\tilde{H}_{12}})=  \frac{\partial}{\partial \phi} f_{23}(L,\phi') \frac{\partial}{\partial \phi} f_{12}(L,\phi)=\frac{d }{d\phi} f_L(\phi)
\end{equation}
and from \eqref{map_circle}-\eqref{map_circleb} that
\begin{equation}
\label{det_J_H}
\det(J_{H_{12}}) = \det(J_{H_{23}}) =-1 \,.
\end{equation}
In addition, we have $\displaystyle \det(J_g)=\frac{\partial \bar{L}}{\partial \theta_1} \frac{\partial \Pi}{\partial \theta_2}-\frac{\partial \bar{L}}{\partial \theta_2} \frac{\partial \Pi}{\partial \theta_1}$, with
\begin{equation}
\label{der_atan2}
\frac{\partial \Pi}{\partial \theta_1} = \frac{-\theta_2}{\theta_1^2+\theta_2^2} \qquad \qquad
\frac{\partial \Pi}{\partial \theta_2} = \frac{\theta_1}{\theta_1^2+\theta_2^2}
\end{equation}
and, from \eqref{invariant_set},
\begin{eqnarray*}
\frac{\partial \bar{L}}{\partial \theta_1} & = & \frac{1}{L} (l_1 l_2 \sin \theta_1 + l_1 l_3 \sin(\theta_1+\theta_2)) \,, \\
\frac{\partial \bar{L}}{\partial \theta_2} & = & \frac{1}{L} (l_2 l_3 \sin \theta_2 + l_1 l_3 \sin(\theta_1+\theta_2))\,.
\end{eqnarray*}
This yields
\begin{equation}
\label{det_J_g}
\begin{split}
\det(J_g) & =\frac{-1}{L (\theta_1^2+\theta_2^2)} \left( l_1 l_2 \theta_1 \sin \theta_1 + l_2 l_3 \theta_2 \sin \theta_2 + l_1 l_3 (\theta_1+\theta_2) \sin(\theta_1+\theta_2) \right) \\
& \leq 0\,,
\end{split}
\end{equation}
where the inequality follows from \eqref{inequa_param} when \eqref{T1}-\eqref{T2}-\eqref{T3} is satisfied (the equality holds only if $(\theta_1,\theta_2)\in \{(0,0),(0,\pi),(\pi,0),(\pi,\pi)\}$). Injecting \eqref{det_J_H_tilde}, \eqref{det_J_H}, and \eqref{det_J_g} into \eqref{det_Jacob}, we obtain
\begin{equation}
\label{sign_orientation}
\frac{d }{d\phi} f_L(\phi) = \det(J_g)(\theta''_1,\theta''_2)/\det(J_g)(\theta_1,\theta_2)> 0
\end{equation}
for all $\phi \in \mathbb{S}$ such that $g^{-1}(L,\phi) \notin \{(0,0),(0,\pi),(\pi,0),(\pi,\pi)\}$. Since $f_L$ is continuous on $\mathbb{S}$, it is an orientation-preserving map on the circle.
\end{proof}

The map $f_L$ preserves the orientation of $\mathbb{S}$, so that we can capture its behavior by the rotation number
\begin{equation*}
\rho(L)=\frac{1}{2\pi} \lim_{n \rightarrow \infty} \frac{(F_L)^n(\phi)}{n}\,, \quad \phi \in \mathbb{S}\,,
\end{equation*}
where $F_L:\mathbb{R} \to \mathbb{R}$ is the \textsl{lifting} of $f_L$, i.e. $F_L$ is a continuous function that satisfies $F_L(\phi) \bmod 2\pi = f_L(\phi \bmod 2 \pi)$. Since it follows from \eqref{det_J_g} and \eqref{sign_orientation} that $f_L \in C^1$ , the rotation number is well-defined and does not depend on $\phi$ (see e.g. \cite{Guckenheimer}, Proposition 6.2.1). When the rotation number is rational, the map has at least one periodic orbit (see e.g. \cite{Guckenheimer}, Proposition 6.2.4). The following result shows that, in our case, all the orbits of $f_L$ are periodic when the rotation number is rational.

\begin{lemma}
\label{lem_all_orbits}
Assume \eqref{T1}-\eqref{T2}-\eqref{T3} is satisfied. If the map $f_L:\mathbb{S} \to \mathbb{S}$ (see \eqref{map_circle_comp}) is characterized by a rotation number $\rho(L) \in \mathbb{Q}$, then all the orbits are periodic with the same period, i.e. there exists $N \in \mathbb{N}$ such that $(f_L)^N=\texttt{Id}$.
\end{lemma}
\begin{proof}
We first introduce the measure
\begin{equation}
\label{inv_measure}
\mu([\phi_a,\phi_b]) = \int_{\phi_a}^{\phi_b} |\det (J_g^{-1}(L,\phi))|\, d\phi
\end{equation}
where $J_g$ is the Jacobian matrix of $g$ defined in \eqref{change_var}. Note that $\mu([\phi_a,\phi_b])=0 \Leftrightarrow \phi_a=\phi_b$. The measure $\mu$ is invariant with respect to $f_L$, since we have
\begin{equation*}
\begin{split}
\mu([f_L(\phi_a),f_L(\phi_b)]) & = \int_{f_L(\phi_a)}^{f_L(\phi_b)} |\det (J_g^{-1}(L,\phi))|\, d\phi \\
& = \int_{\phi_a}^{\phi_b} |\det (J_g^{-1}(L,f_L(\phi)))| \left(\frac{d f_L}{d\phi}\right) \, d\phi \\
& =  \int_{\phi_a}^{\phi_b} |\det (J_g^{-1}(L,\phi))|\, d\phi \\
& = \mu([\phi_a,\phi_b])
\end{split}
\end{equation*}
where we used \eqref{det_Jacob}, \eqref{det_J_H_tilde}, and \eqref{det_J_H}.\\
Since $f_L$ is an orientation-preserving map (Lemma \ref{lem_orient}) and has a rational rotation number by the hypothesis of the theorem, every orbit is periodic (of period $N$) or converges to a periodic orbit (see e.g. \cite{Guckenheimer}). Assume that the second case is possible, i.e. there exist $\phi,\phi^* \in \mathbb{S}$ with $\phi \neq \phi^*$ such that
\begin{equation*}
\lim_{k \rightarrow \infty} (f_L)^{kN}(\phi)=\phi^* \qquad (f_L)^N(\phi^*)=\phi^*\,.
\end{equation*}
We have $\mu([\phi,\phi^*]) \neq 0$ and
\begin{equation*}
\lim_{k \rightarrow \infty} \mu([(f_L)^{kN}(\phi),(f_L)^{kN}(\phi^*)])=\lim_{k \rightarrow \infty} \mu([(f_L)^{kN}(\phi),\phi^*]) = 0\,.
\end{equation*}
This contradicts the invariance of $\mu$. Then every orbit must be periodic.
\end{proof}
\begin{remark}
Since the map $f_L$ has a non-singular invariant measure $\mu$, it is conjugate to a pure rotation $\varphi_{\rho(L)}:\phi \mapsto \phi+2\pi \rho(L)$, i.e. there exists a conjugating map $h:\mathbb{S} \to \mathbb{S}$ such that $h \circ f_L = \varphi_{\rho(L)} \circ h$. The conjugating map is given by $h(\phi)=\mu([0,\phi])$ \cite{cornfeld2012ergodic}. It follows that, for all $\phi \in \mathbb{S}$,
\begin{equation*}
\rho(L) = \frac{1}{2\pi} \left(\varphi_{\rho(L)} \circ h (\phi) - h(\phi) \right) = \frac{1}{2\pi} \Big(h \circ f_L(\phi) - h(\phi) \Big) = \frac{1}{2\pi} \mu([\phi,f_L(\phi)]) \,.
\end{equation*}
Then \eqref{inv_measure} implies that the rotation number is given by
\begin{equation*}
\rho(L) = \frac{1}{2\pi} \int_\phi^{f_L(\phi)} |\det (J_g^{-1}(L,\phi'))| \, d\phi'
\end{equation*}
for all $\phi \in \mathbb{S}$. \hfill $\diamond$
\end{remark}

Finally, the map $f_L$ satisfies the following important property.
\begin{lemma}
\label{lam_f_monot}
Assume \eqref{T1}-\eqref{T2}-\eqref{T3} is satisfied. Then, the map $f_L:\mathbb{S} \to \mathbb{S}$ (see \eqref{map_circle_comp}) satisfies
\begin{equation*}
\frac{\partial }{\partial L} f(L,\phi)
\begin{cases}
\geq 0 & \textrm{if } l_2 \geq \max\{l_1,l_3\}\,. \\
\leq 0 & \textrm{if } l_2 \leq \min\{l_1,l_3\}\,.
\end{cases}
\end{equation*}
If $l_1\neq l_2$ or $l_2\neq l_3$, the equality holds only for $\phi \in \{\pi/2,-\pi/2,f_{12}^{-1}(L,0),f_{12}^{-1}(L,\pi)\}$.
\end{lemma}
\begin{proof}
Using \eqref{map_circle}, we have
\begin{equation*}
\left. \frac{\partial }{\partial \gamma} f_{12}(g(\gamma \theta_1,\gamma \theta_2))\right|_{\gamma=1}=
\left.\frac{\partial \Pi}{\partial \theta_1}\right|_{(-\theta_1,\theta_2+ \Delta(\theta_1))} (-\theta_1) + \left. \frac{\partial \Pi}{\partial \theta_2} \right|_{(-\theta_1,\theta_2+ \Delta(\theta_1))} \left(\theta_2+\theta_1 \frac{d \Delta}{d \theta_1}\right)
\end{equation*}
with
\begin{equation*}
\Delta(\theta_1) \triangleq \textrm{sign}(\theta_1) \, 2 \arccos \left(\frac{l_2+l_1 \cos \theta_1}{\sqrt{l_1^2+l_2^2+2 l_1 l_2 \cos\theta_1}}\right)
\end{equation*} 
and \eqref{der_atan2} leads to
\begin{equation}
\label{der_f_gamma}
\left. \frac{\partial }{\partial \gamma} f_{12}(g(\gamma \theta_1,\gamma \theta_2))\right|_{\gamma=1}=  \frac{|\theta_1|}{\theta_1^2+(\theta_2+\Delta(\theta_1))^2} G(\theta_1)
\end{equation}
with
\begin{equation*}
G(\theta_1) \triangleq \textrm{sign}(\theta_1) \left(\Delta(\theta_1) - \theta_1 \frac{d \Delta}{d \theta_1}\right)\,.
\end{equation*}
We have
\begin{equation}
\label{der_G}
\frac{dG}{d\theta_1} = -|\theta_1| \frac{d^2 \Delta}{d \theta_1^2} = -|\theta_1| \frac{2 l_1 l_2 \sin \theta_1 (l_1^2-l_2^2)}{(l_1^2+l_2^2+2 l_1 l_2 \cos \theta_1)^2}
\end{equation}
where we omitted the lengthy (but straightforward) computation of $d^2 \Delta/d\theta_1^2$.
Since $G$ is continuous and $G(0)=\Delta(0)=0$, \eqref{der_G} implies that, for all $\theta_1 \in(-\pi,\pi]$, $G(\theta_1) \leq 0$ if $l_1\geq l_2$ and $G(\theta_1) \geq 0$ if $l_1 \leq l_2$. Equivalently, it follows from \eqref{der_f_gamma} that
\begin{equation*}
\left. \frac{\partial }{\partial \gamma} f_{12}(g(\gamma \theta_1,\gamma \theta_2))\right|_{\gamma=1}
\begin{cases}
\leq 0 & \textrm{if } l_1\geq l_2 \\
\geq 0 & \textrm{if } l_1\leq l_2
\end{cases}
\end{equation*}
and \eqref{inequa_param} implies
\begin{equation}
\label{der_L12}
\frac{\partial }{\partial L} f_{12}(L,\phi)
\begin{cases}
\geq 0 & \textrm{if } l_1\geq l_2\,. \\
\leq 0 & \textrm{if } l_1\leq l_2\,.
\end{cases}
\end{equation}
If $l_1 \neq l_2$, the equality holds only if $\sin \theta_1 =0 $. Since \eqref{T1} and \eqref{T3} are satisfied, we have $\theta_1<\pi$, which implies that the equality holds only if $\theta_1=0$ (i.e. $\phi\in\{-\pi/2,\pi/2\})$.\\
It follows on similar lines that
\begin{equation}
\label{der_L23}
\frac{\partial }{\partial L} f_{23}(L,\phi)
\begin{cases}
\geq 0 & \textrm{if } l_2\geq l_3 \\
\leq 0 & \textrm{if } l_2\leq l_3
\end{cases}
\end{equation}
and, if $l_2 \neq l_3$, the equality holds only if $\theta_2=0$ (i.e. $\phi\in\{0,\pi\}$).\\
Finally, $\det(J_{\tilde{H}_{23}}) = \det(J_g) \, \det(J_{H_{23}}) \, \det(J_g)^{-1}$ and \eqref{det_J_H_tilde}-\eqref{det_J_H}-\eqref{det_J_g} imply that
\begin{equation}
\label{der_phi23}
\frac{\partial f_{23}}{\partial \phi}  < 0\,
\end{equation}
and the result follows from
\begin{equation*}
\frac{\partial f}{\partial L}(L,\phi)=\frac{\partial f_{23}}{\partial L}(L,f_{12}(L,\phi)) + \frac{\partial f_{23}}{\partial \phi} (L,f_{12}(L,\phi)) \, \frac{\partial f_{12}}{\partial L}(L,\phi)
\end{equation*}
with the inequalities \eqref{der_L12},\eqref{der_L23}, and \eqref{der_phi23}.
\end{proof}

\subsection{Proof of the main result}
\label{subsec:proof_main}

We are now in position to prove Theorem \ref{theo_dense}, or equivalently, Theorem \ref{main_theorem}.
\begin{proof}
The orbits of $\langle H_{23} \rangle \circ \langle H_{12}\rangle$ are dense in $\Gamma(L)$ if and only if the orbits of $f_L$ are dense in $\mathbb{S}$.\\
It follows from Lemma \ref{lem_orient} that $f_L$ is an orientation-preserving map on the circle. It is clear from \eqref{det_J_g} and \eqref{sign_orientation} that $\log (d f_L/d\phi)$ has bounded variation. Hence Denjoy's theorem implies that the orbits of $f_L$ are dense if and only if the rotation number $\rho(L)$ is irrational (see e.g. \cite{Guckenheimer}, Theorem 6.2.5).\\
Consider a set of parameters $l_1$,$l_2$,$l_3$ satisfying \eqref{cond_main_theorem} (with $l_1\neq l_2$ or $l_2\neq l_3$) and admissible values $L$ satisfying \eqref{T1}-\eqref{T2}-\eqref{T3}. We assume that $\rho(L_a)=\rho(L_b) \in \mathbb{Q}$ for some $L_a < L_b$. Lemma \ref{lem_all_orbits} implies that
\begin{equation}
\label{contradict_theor}
(f^{(L_a)})^N(\phi)=(f^{(L_b)})^N(\phi)=\phi
\end{equation}
for all $\phi \in \mathbb{S}$ and for some $N \in \mathbb{N}$. For some $\phi \notin \{\pi/2,-\pi/2,f_{12}^{-1}(L,0),f_{12}^{-1}(L,\pi)\}$, it follows from Lemma \ref{lam_f_monot} that
\begin{equation}
\label{property_monot}
(f^{(L_a)})(\phi) < (f^{(L_b)})(\phi)
\end{equation}
provided that $l_2 \geq l_1$ and $l_2 \geq l_3$. Since $(f^{(L_a)})^{N-1}$ is orientation-preserving, we have
\begin{equation*}
(f^{(L_a)})^N(\phi) < ((f^{(L_a)})^{N-1} \circ (f^{(L_b)}))(\phi) < (f^{(L_b)})^N(\phi)
\end{equation*}
where the last inequality follows again from \eqref{property_monot}. If $l_2 \leq l_1$ and $l_2 \leq l_3$, we obtain similarly $(f^{(L_a)})^N(\phi)>(f^{(L_b)})^N(\phi)$. This contradicts \eqref{contradict_theor}, so that $\rho(L_a)=\rho(L_b) \in \mathbb{Q}$ implies $L_a = L_b$. Then, for given values $l_1$,$l_2$,$l_3$, the function $L \mapsto \rho(L)$ is nowhere constant. In addition, since $f_L$ does not admit a fixed point (i.e. two successive pops of different bars cannot leave the configuration unchanged), it follows from Lemma \ref{lam_f_monot} that $L \mapsto \rho(L)$ is strictly monotone (increasing or decreasing) (see e.g. \cite{Guckenheimer}, Proposition 6.2.3). In addition, $\rho(L)$ is a continuous function of $L$ (see e.g. \cite{Guckenheimer}, Proposition 6.2.2), so that the inverse $\rho^{-1}$ is absolutely continuous \cite{Lusin}. Then, $\rho^{-1}(E)$ is a zero measure set if $E$ is a zero measure set (Lusin's condition). With $E=\mathbb{Q}$, one obtains that $\rho(L)$ is rational only on a zero measure set of admissible values $L$. This concludes the proof.
\end{proof}

\section{Conclusion}
\label{conclu}

Motivated by longstanding open questions in computational geometry, we have studied 
a simple framework, the four-bar linkage with one fixed bar, and its behavior when a long series of ``pops'' are alternatively applied to two mobile vertices. Our main contribution is to show that, 
under particular conditions, the dynamics of the four-bar linkage under pops is 
topologically equivalent to an orientation-preserving map of the circle with an 
irrational rotation number, so that each orbit densely fills the configuration space. 

A general statement on the behavior of the four-bar linkage under pops is summarized in Conjecture~\ref{conjecture} which, if true, would have important consequences in the theory of mechanical linkages. In the context of dynamical systems theory, the conjecture can be recast as follows.
\begin{customconj}{1bis}
For almost all parameters $l_1,l_2,l_3,L$, every orbit of the map \eqref{syst}-\eqref{systb}, with an initial condition $(\theta_1,\theta_2)\in \Gamma(L)$, is dense in a connected component of $\Gamma(L)$.
\end{customconj}

Although additional conditions on the length of the bars were needed to establish the results of this paper, 
numerical simulations suggest that they are conservative and we suspect that they could be removed by 
obtaining additional properties of the rotation number. 
On top of these restrictions on the numerical values of the parameters, 
our results only hold for configuration spaces that are connected (i.e. non-Grashof cases).
When configuration spaces are not connected (Grashof cases), other results and techniques are required for a complete proof of the conjecture (see e.g. \cite{Izmestiev}).

\section*{Acknowledgement}

We thank F. Gonze and P.-Y. Chevalier for their help in the proof of Proposition \ref{prop_non_ergo}. We are also thankful to Ivan Izmestiev to have pointed out his work and relevant references on the four bars problem.

\bibliography{pops}

\begin{thebibliography}{10}

\bibitem{aloupis-07}
{\sc G.~Aloupis, B.~Ballinger, P.~Bose, M.~Damian, E.~D. Demaine, M.~L.
  Demaine, R.~Flatland, F.~Hurtado, S.~Langerman, J.~O'Rourke, P.~Taslakian,
  and G.~T. Toussaint}, {\em Vertex pops and popturns}, in Proceedings of the
  19th Canadian Conference in Computational Geometry, Ottawa, ON, Canada,
  August 2007, pp.~137--140.

\bibitem{ballinger-03}
{\sc B.~Ballinger}, {\em Length-Preserving Transformations on Polygons}, PhD
  thesis, University of California, Davis, California, 2003.

\bibitem{benoist2004iteration}
{\sc Y.~Benoist and D.~Hulin}, {\em It{\'e}ration de pliages de
  quadrilat{\`e}res}, Inventiones mathematicae, 157 (2004), pp.~147--194.

\bibitem{cornfeld2012ergodic}
{\sc I.~P. Cornfeld, S.~V. Fomin, and Y.~G. Sinai}, {\em Ergodic theory},
  vol.~245, Springer Science \& Business Media, 2012.

\bibitem{darboux1879}
{\sc G.~Darboux}, {\em De l'emploi des fonctions elliptiques dans la
  th{\'e}orie du quadrilat{\`e}re plan}, Bulletin des sciences
  math{\'e}matiques et astronomiques, 3 (1879), pp.~109--128.

\bibitem{Demaine-08}
{\sc E.~D. Demaine, B.~Gassend, J.~O'Rourke, and G.~T. Toussaint}, {\em All
  polygons flip finitely\dots\ right?}, in Surveys on Discrete and
  Computational Geometry: Twenty Years Later, J.~Goodman, J.~Pach, and
  R.~Pollack, eds., American Mathematical Society, 2008, pp.~231--255.

\bibitem{dumitrescu-10}
{\sc A.~Dumitrescu and E.~Hilscher}, {\em On convexification of polygons by
  pops}, Discrete Mathematics, 310 (2010), pp.~2542--2545.

\bibitem{erdos-35}
{\sc P.~Erd\H{o}s}, {\em Problem 3763}, American Mathematical Monthly, 42
  (1935), p.~627.

\bibitem{fevens-01}
{\sc T.~Fevens, A.~Hernandez, A.~Mesa, P.~Morin, M.~Soss, and G.~Toussaint},
  {\em Simple polygons with an infinite sequence of deflations}, Contributions
  to Algebra and Geometry, 42 (2001), pp.~307--311.

\bibitem{Gru-95}
{\sc B.~Gr\H{u}nbaum}, {\em How to convexify a polygon}, Geombinatorics, 5
  (1995), pp.~24--30.

\bibitem{Gru-01}
{\sc B.~Gr\H{u}nbaum and J.~Zaks}, {\em Convexification of polygons by flips
  and by flipturns}, Discrete Mathematics, 241 (2001), pp.~333--342.

\bibitem{Guckenheimer}
{\sc J.~Guckenheimer and P.~Holmes}, {\em Nonlinear Oscillations, Dynamical
  Systems, and Bifurcations of Vector Fields}, vol.~42, Springer Science \&
  Business Media, 1983.

\bibitem{hubard-10}
{\sc I.~Hubard and P.~Taslakian}, {\em Deflating polygons to the limit}, in
  Proceedings of the 22nd Canadian Conference in Computational Geometry,
  Windsor, MB, Canada, August 2010, pp.~67--70.

\bibitem{Izmestiev}
{\sc I.~Izmestiev}, {\em Deformation of quadrilaterals and addition on elliptic
  curves}.
\newblock http://arxiv.org/abs/1501.07157, 2015.

\bibitem{kapovich-96}
{\sc M.~Kapovich and J.~Millson}, {\em The symplectic geometry of polygons in
  {E}uclidean space}, Journal of Differential Geometry, 44 (1996),
  pp.~479--513.

\bibitem{complexGeometry-13}
{\sc G.~Khimshiashvili}, {\em Complex geometry of polygonal linkages}, Journal
  of Mathematical Sciences, 189 (2013), pp.~132--149.

\bibitem{mccarthy-06}
{\sc J.~M. McCarthy}, {\em Geometric Design of Linkages}, vol.~11, Springer
  Science \& Business Media, 2006.

\bibitem{mermoud-00}
{\sc O.~Mermoud and M.~Steiner}, {\em Visualisation of configuration spaces of
  polygonal linkages}, Journal of Geometry and Graphics, 4 (2000),
  pp.~147--157.

\bibitem{millet-94}
{\sc K.~C. Millet}, {\em Knotting of regular polygons in 3-space}, Journal of
  Knot Theory and its Ramifications, 3 (1994), pp.~263--278.

\bibitem{muller-96}
{\sc M.~Muller}, {\em A novel classification of planar four-bar linkages and
  its application to the mechanical analysis of animal systems}, Philosophical
  Transactions: Biological Sciences, 351 (1996), pp.~pp. 689--720.

\bibitem{nagy-39}
{\sc B.~S. Nagy}, {\em Solution of problem 3763}, American Mathematical
  Monthly, 46 (1939), pp.~176--177.

\bibitem{roberts-75}
{\sc S.~Roberts}, {\em On three-bar motion in plane space}, Proceedings of the
  London Mathematical Society, 1 (1875), pp.~14--23.

\bibitem{toussaint-05}
{\sc G.~T. Toussaint}, {\em The {E}rd{\H{o}}s-{N}agy theorem and its
  ramifications}, Computational Geometry: Theory and Applications, 31 (2005),
  pp.~219--236.

\bibitem{Lusin}
{\sc A.~Villani}, {\em On {L}usin's condition for the inverse function},
  Rendiconti del Circolo matematico di Palermo, 33 (1984), pp.~331--335.

\end{thebibliography}
\bibliographystyle{siam}

\end{document}